\documentclass[12pt,article]{amsart}
\usepackage{amsmath,ytableau,verbatim}

\usepackage{xcolor}

\newtheorem{theorem}{Theorem}[section]
\newtheorem{lemma}[theorem]{Lemma}
\newtheorem{proposition}[theorem]{Proposition}
\newtheorem{corollary}[theorem]{Corollary}

\theoremstyle{definition}

\newtheorem{example}[theorem]{Example}

\newtheorem{remark}[theorem]{Remark}

\newcommand{\V}{\mathbb{V}}
\newcommand{\W}{\mathbb{W}}
\newcommand{\K}{\mathcal{K}}
\newcommand{\A}{\mathcal{A}}
\newcommand{\R}{\mathbb{R}}
\newcommand{\C}{\mathbb{C}}
\newcommand{\Sb}{\mathbb{S}}

\newcommand{\U}{\mathrm{U}}
\newcommand{\SO}{\mathrm{SO}}
\newcommand{\OO}{\mathrm{O}}

\newcommand{\Val}{\mathrm{Val}}
\newcommand{\Area}{\mathrm{Area}}
\newcommand{\GL}{\mathrm{GL}}
\newenvironment{proofof}[1]{\noindent{\it Proof of
#1.}}{\hfill$\square$\\\mbox{}}

\usepackage{amscd,amssymb}

\begin{document}

\title[Unitary equivariant tensor valuations]{Dimension of the space of unitary equivariant translation invariant tensor valuations} 

\author{K.J. B\"or\"oczky, M. Domokos, G. Solanes}\thanks{ The first two authors were partially supported by National Research, Development and Innovation Office,  NKFIH K 119934, 
K 132002, KH 129630, ANN 121 649. The third author was partially supported by FEDER-MINECO grant PGC2018-095998-B-I00, and the Serra H\'unter Programme.}

\address{Alfr\'ed R\'enyi Institute of Mathematics, Re\'altanoda u. 13-15, H-1053 Budapest, Hungary and CEU, N\'ador u 9, H-1051 Budapest, Hungary}
\address{Alfr\'ed R\'enyi Institute of Mathematics, Re\'altanoda u. 13-15, H-1053 Budapest, Hungary}
\address{Departament de Matem\`atiques\\ Universitat Aut\`onoma de Barcelona\\08193 Bellaterra\\ Spain} 

\begin{abstract} 
Following the work of Semyon Alesker in the scalar valued case and of Thomas Wannerer in the vector valued case,
the dimensions of the spaces of continuous translation invariant and unitary equivariant tensor valuations are computed. In addition, a basis in the vector valued case is presented.
\end{abstract} 
\maketitle

\section{Introduction}

For a real vector space $\V$  of finite  dimension  and an abelian semigroup $(\A,+)$, we write $\K(\V)$ to denote the space of convex bodies in $\V$ (i.e., compact convex sets) equipped with the Hausdorff metric.  
We call an operator $Z:\K(\V)\to\A$ a \emph{valuation} if 
$$
Z(K\cup L)+Z(K\cap L)=Z(K)+Z(L)
$$
holds for any $K,L\in\K(\V)$ satisfying that $K\cup L\in\K(\V)$. Typical choices for the semigroup $\A$ are the field of real numbers $\mathbb R$, or the vector space $\mathbb V$ itself, and more generally the space $\Sb^d(\V)$
of symmetric rank $d$ tensors of $\V$.  Also $\A=\K(\V)$ equipped with Minkowski addition leads to many interesting valuations.

One of the principal goals in the theory of valuations is to obtain characterizations of known operators as the only valuations satisfying certain simple geometric and topological properties.
The fundamental result in this direction goes back to 1952, when Hadwiger proved that, for $\V=\R^n$, the linear combinations of intrinsic volumes are the only continuous real-valued valuations being invariant under rigid motions of $\R^n$ (see \cite{Had57}). 

Hadwiger's result can be generalized in different directions. One of them is to change the group acting on $\K(\V)$ and classify the continuous real-valued translation invariant valuations that are invariant under the linear action of some group $G$. The space of such valuations is finite-dimensional precisely when $G$ acts transitively on the unit sphere \cite{Ale00}. For the first nontrivial case $G=\U(m)$, this was achieved by Alesker \cite{Ale03} and refined by Bernig and Fu \cite{Ber-Fu}. After this breakthough, several other groups have been succesfully studied. We refer the reader to \cite{AbW15,Ber09,Ber11,bernig.fu.solanes,Ber-Sol,Fu06_2,wannerer14} and references therein for some results in this direction. 

Another important generalization of Hadwiger's theorem consists of  
changing the target space $\A$. The case $\A=\Sb^d(\V)$ of tensor-valued valuations is of particular interest and has been thoroughly studied, specially under equivariance  assumptions with respect to orthogonal and special linear groups (see e.g. \cite{ABDK,  bernig.hug, HaP0, HaP,hug.schneider.schuster}). The space of $\U(m)$-equivariant valuations is much less understood, and is the object of the present paper.

Other current research directions in valuation theory include the following. Real-valued and tensor-valued valuations defined on lattice polytopes have been studied in \cite{boroczky.ludwig, LuS17}. A very active area is the study of valuations taking values in the space of convex bodies  (see e.g., \cite{ludwig02,ludwig05} and the references in \cite{ludwig.survey}). Also, important results on valuations defined in several function spaces have been recently obtained (cf. e.g., \cite{Ale17+,boroczky.ludwig,colesanti.ludwig.mussnig,ludwig.survey2}).

\medskip In this paper we begin the study of unitary-equivariant valuations on  complex vector spaces. To state our results precisely, let us introduce some notation.
We denote the space of continuous translation invariant  $\R$-valued 
valuations on $\K(\V)$ by
 $\mathrm{Val}=\Val(\V)$. The subspace of $k$-homogeneous valuations (i.e. such that $Z(\lambda K)=\lambda^k Z(K)$ for any convex body $K$ and $\lambda> 0$) is denoted by 
$\mathrm{Val}_k$. 
Given a  linear action of a closed subgroup $G\subset \GL(\V)$ on  a finite dimensional $\R$-vector space $\W$, 
we say that a valuation $Z:\K(\V)\to  \W$ is \emph{$G$-equivariant} if $Z(\varphi(K))=\varphi Z(K)$ holds for any $\varphi\in G$ and  $K\in \mathcal{K}(\V)$. 
The space of $\W$-valued continuous translation invariant $G$-equivariant valuations is 
 naturally identified with the  subspace   
$(\mathrm{Val}\otimes \W)^G$ of $G$-invariants in $\mathrm{Val}\otimes \W$  
 (the symbol $\otimes$ is used in this paper for tensor products over $\R$).  

We will focus on continuous translation invariant and $\U(m)$-equivariant tensor-valued valuations on $\K(\R^{2m})$ for $m\ge 2$. So in our case,  $\V$ will be $\C^m$, viewed as the real vector space $\R^{2m}$.  The group $G$ will be the unitary group $\U(m)$ with its defining action on $\V=\C^m$, and $\W=\Sb^d(\R^{2m})$ will be the $d^{\mathrm{th}}$ symmetric tensor power over $\R$ of $\V$.   
For the homogenity degree $k=0,\ldots,2m$ of a valuation, we set $\ell=\min\{k,2m-k\}$, 
and write $\lfloor \frac{\ell}2\rfloor$ for the lower integer part of $\frac{\ell}2$. 
In the scalar valued case, Alesker \cite{Ale03} proved that 
${\rm dim}_\R\mathrm{Val}_k^{\U(m)}=1+\lfloor \frac{\ell}2\rfloor$, and
he provided two different sets of bases for $\mathrm{Val}_k^{\U(m)}$.
In the vector valued case, Wannerer \cite{wannerer} obtained 
${\rm dim}_\R(\mathrm{Val}_k\otimes   \R^{2m})^{\U(m)}=2\lfloor \frac{\ell}2\rfloor$ 
 (see also \cite[Theorem  6.14]{schuster}). 

Our main result is the determination of the dimension of the space of 
$\U(m)$-equivariant tensor valued valuations of all ranks:

\begin{theorem}\label{thm:dimensions}
\label{dimensions}
For $m\geq 2$, $k=0,\ldots, 2m $  and $d\geq 0$, using  the notation  
$f:=\lfloor \frac d2 \rfloor$ and $\ell:=\min\{k,{2m}-k\}$, the dimension of
$(\mathrm{Val}_k\otimes  \Sb^d(\R^{2m}))^{\U(m)}$ is as follows: 
$$
\begin{array}{cc||c}
 & & \dim_{\R}((\mathrm{Val}_k\otimes \Sb^d(\R^{2m}))^{\U(m)})\\ \hline \hline 
d=0 &  & 1+\lfloor \frac{\ell}2\rfloor \\
d=2f>0, & \ell=0 & 1 \\
d=2f>0, & 1\le \ell<m & 3\ell f^2+2\lfloor \frac{\ell}{2} \rfloor-2f^2+2f+1 \\
d=2f>0, & \ell=m & 3mf^2+2\lfloor \frac m2 \rfloor -3f^2+2f+1 \\
d=2f+1, & \ell=0 & 0 \\
d=2f+1, & 1\le \ell<m & 3\ell f^2+3\ell f+2\lfloor \frac{\ell}{2}\rfloor-2f^2 \\
d=2f+1, & \ell=m &  3mf^2+3mf + 2\lfloor \frac m2\rfloor-3f^2-f
\end{array} 
$$
\end{theorem}

To compute these dimensions, we first obtain the multiplicity in $\C\otimes \Val_k$ 
of each irreducible $\U(m)$-module direct summand of  $\C\otimes \Sb^d(\R^{2m})$ 
(see Theorem \ref{prop:table_5}). This intermediate result may also be of independent interest.

Once its dimension is known, a natural goal is to determine a basis for $(\mathrm{Val}_k\otimes \Sb^d(\R^{2m}))^{\U(m)}$.
So far this has been only known for rank $d=0$ by Alesker \cite{Ale03}. Here we construct such a basis 
in the vector valued  case (i.e. for $d=1$). 
To this end, we use the area measures introduced by Wannerer \cite{wannerer} (see Section~\ref{secVecbasis} for definitions). We write $S(\V)$ to denote the unit sphere in a euclidean vector space $\V$,
and ${\rm Area}(\V)$ to denote the space of smooth area measures, which is a certain class of translation invariant valuations taking values in the space of signed measures  of $S(\V)$. To each $\Psi\in{\rm Area}(\V)$, one can assign  the smooth vector valued valuation $C(\Psi)$ defined by
$$
C(\Psi)(K)=\int_{S(\V)}u d\Psi(K,du)
$$
for any convex body $K$. 

For $\V=\C^m$, Wannerer \cite{wannerer} gave a complete description of the space $\Area^{\U(m)}$ of $\U(m)$-equivariant area measures. In particular, he introduced a certain family $\Delta_{k,q}\in\Area^{\U(m)}$ with specially nice properties (see Section~\ref{secVecbasis}).

\begin{theorem}\label{thm:Vecbasis}Consider  the $\C$-vector space structure on $\mathrm{Val}_k\otimes_{\R}\C^m$ 
given by $\alpha(Z\otimes u)=Z\otimes (\alpha u)$ for $\alpha\in \C$, $u\in \C^m$, and 
$Z\in \mathrm{Val}_k$. Then, for $m\geq 2$, a  $\C$-vector space basis of $(\mathrm{Val}_k\otimes_{\R}\C^m)^{\U(m)}$ is given by
the family $C(\Delta_{k,q})$ where  $0\leq k<2 m$ and $\max(0,k- m)< q\leq \frac{k}2$. 
\end{theorem}

 One obtains the following immediate corollary: 

\begin{corollary}
For $m\geq 2$, an $\R$-vector space 
basis of $(\Val_k\otimes \R^{2m})^{\U(m)}$ is 
$$\{C(\Delta_{k,q}), \   \sqrt{-1}\cdot C(\Delta_{k,q})\ \mid\  0,k-m< q\leq  {k \over 2}\}.$$  
\end{corollary} 

\section{Branching rules}\label{sec:modification} 

Let us recall some background material from representation theory. 
For technical reasons we shall work with complex representations.  
So given a compact Lie group $G$, by a {\it $G$-module} we shall mean a   finite dimensional complex vector space $V$ endowed with an action of $G$ via $\C$-linear transformations, such that the corresponding group homomorphism $G\to \GL(V)$ is smooth.  
Write $\overline{V}$ for the isomorphism class of $V$. The set of isomorphism classes of $G$-modules is a commutative monoid with addition given by $\overline{V}+\overline{W}=\overline{V\oplus W}$. 
The {\it Grothendieck group} of this monoid consists of formal differences of isomorphism classes of $G$-modules. It is a free abelian group $R_G$ freely generated by $\tau_G$, the set of isomorphism classes of irreducible $G$-modules. 
In fact $R_G$ is a ring, called the {\it representation ring of $G$}, with multiplication 
given by $\overline{V}\cdot \overline{W}=\overline{V\otimes W}$. 
For ease of language or notation we shall frequently 
identify the isomorphism class $\overline{V}\in R_G$ with a $G$-module $V$ representing the isomorphism class $\overline{V}$.  
We shall denote by 
$V\downarrow^G_H$ the restriction of the $G$-module $V$ to a subgroup $H$ of $G$.

\bigskip
Let us briefly sketch the strategy in the proof of Theorem \ref{thm:dimensions}.
We will consider the complexifications 
\begin{align*}\mathrm{Val}_{k,\C}&:=\C\otimes\mathrm{Val}_k \\ 
\Sb_{\C}^d(\R^{2m})&:=\C\otimes \Sb^d(\R^{2m}).
\end{align*} 
Note that while $\R^{2m}$ is an irreducible real representation of $\U(m)$, its complexification $\Sb_{\C}^1(\R^{2m})=\C\otimes \R^{2m}$ is the direct sum of the natural 
$\U(m)$-module $\C^m$ and its dual (over $\C$). It follows that $\Sb_{\C}^d(\R^{2m})$ is a self-dual $\U(m)$-representation.

Clearly
\begin{equation}\label{eq:real_complex}
 \dim_\R(\Val_k\otimes\Sb^d(\R^{2m}))^{\U(m)}= \dim_\C(\Val_{k,\C}\otimes_\C\Sb^d_\C(\R^{2m}))^{\U(m)}.
\end{equation}
Using that $\Sb^d_\C(\R^{2m})$ is a self-dual $\U(m)$-representation, we obtain the standard isomorphism
\begin{equation}\label{eq:tensor_hom}
(\Val_{k,\C}\otimes_\C\Sb^d_\C(\R^{2m}))^{\U(m)}\cong 
\mathrm{Hom}_{\U(m)}(\Sb^d_\C(\R^{2m}),\mathrm{Val}_{k,\C}).
\end{equation}

Starting from known decompositions into irreducible $\SO(2m)$-modules, and restricting those to $\U(m)$, we will determine the decomposition into irreducible summands of the $\U(m)$-modules $\Val_{k,\C}$ and $\Sb_\C^d(\R^{2m})$. Combining \eqref{eq:real_complex} and \eqref{eq:tensor_hom}, and using Schur's Lemma will then yield the dimension of $(\mathrm{Val}_k\otimes \Sb^d(\R^{2m}))^{\U(m)}$.

Next we turn to a parametrization of $\tau_{\OO(2m)}$ based on partitions, where $\OO(2m)$ is the full orthogonal group. 
This material can be found for example in \cite{procesi}, \cite{goodman-wallach}, 
\cite{howe-tan-willenbring} (the notation in these sources is different, and they mainly work in the context of complex linear algebraic groups and not with compact Lie groups).  
Set 
\[\Pi_m=\{\lambda=(\lambda_1,\dots,\lambda_m)\mid \lambda_1\ge\cdots \ge \lambda_m\ge 0, 
\quad \lambda_i\in \mathbb{Z}\}.\]
For $p\le m$ we shall treat $\Pi_p$ as a subset of $\Pi_m$, by identifying 
$(\lambda_1,\dots,\lambda_p)\in \Pi_p$ with $(\lambda_1,\dots,\lambda_p,0,\dots,0)$. 
Set 
\[\Pi_m^+=\{\lambda\in \Pi_m\mid \lambda_m>0\}.\]  
Now we have 
\[\tau_{\OO(2m)}=\{[\lambda],\ [\lambda]^\circ, \ [\mu]\mid 
\mu\in \Pi_m^+,\ \lambda\in \Pi_m\setminus \Pi_m^+\}.\]
Here $[\lambda]^\circ=[\lambda]\cdot \varepsilon$, where $\varepsilon$ is the $1$-dimensional $\OO(2m)$-module given by the determinant. Note that \cite[p. 418, Theorem 2]{procesi} labels the elements of $\tau_{\OO(2m)}$ by partitions such that the sum of the lengths of the first two columns of their Young diagram is at most $2m$. 
Denoting by $\lambda'_1$ the length of the first column of the Young diagram of $\lambda$, our $[\lambda]^\circ$ corresponds to the partition whose Young diagram is obtained by replacing the first column of the Young diagram of $\lambda$ by a column of length $2m-\lambda'_1$. For example, for the partition $\lambda=(1^k)=(1,\dots,1)$ (with $k$ components $1$ for some $0\le k\le m$), we have that $[\lambda]$ is the $k^{\mathrm{th}}$ exterior power of the 
natural $\OO(2m)$-module $\C^{2m}$, and   $[\lambda]^\circ$ is the 
$(n-k)^{\mathrm{th}}$ exterior power of the natural $\OO(2m)$-module $\C^{2m}$. 
We mention that the irreducible $\OO(2m)$-modules are all defined over $\R$, that is, they are complexifications of irreducible real $\OO(2m)$-modules. 

Next we turn to the unitary group   $\U(m)=\OO(2m)\cap \GL(\C^m)$, where  $\GL(\C^m)$ is the complex general linear group.
We have 
\[\tau_{\U(m)}=\{\{ \mu;\lambda\}\mid \lambda\in \Pi_p^+,\ 
\mu\in\Pi_q^+, \ p+q\le m\}.\] 
Here $\{ \mu;\lambda\}$ is the irreducible $\U(m)$-module with highest weight 
\[(\lambda_1,\dots,\lambda_p,0,\dots,0,-\mu_q,\dots,-\mu_1).\]   
This is the restriction to the maximal compact subgroup $\U(m)$ of the general linear group $\GL(\C^m)$ of its representation obtained by tensoring by the $(-\mu_1)^{\mathrm{th}}$ power of the determinant representation  the Schur module associated with the partition $(\lambda_1+\mu_1,\lambda_2+\mu_1,\dots)\in \Pi_m$ (cf. \cite[p. 278,  (8.1.3)]{procesi}).     
For example, the natural $\U(m)$-module $\C^m$ is $\{0;1\}$, where 
in order to simplify the notation we write 
$\{j;i\}:=\{(j,0,\dots,0);(i,0,\dots,0)\}$. The dual of the  natural $\U(m)$-module $\C^m$ is 
$\{1;0\}$, and the $k^{\mathrm{th}}$ exterior power of $\C^m$ is $\{(0);(1^k)\}$.

In certain formulae below elements of $R_{\U(m)}$ denoted by $\{ \mu;\lambda\}$ where 
$\lambda\in\Pi_p^+$, $\mu\in\Pi_q^+$, and $p+q>m$  will also occur. They can be expressed as an integral linear combination of elements in $\tau_{\U(m)}$ by a repeated application of the following modification rule given by King \cite[p. 433]{king}, see also \cite[Section 3]{black-etal}:   
 
Set $h=p+q-m-1$. In case it is possible to remove a boundary strip of $h$ boxes from the Young diagram of $\lambda$, starting at the foot of the first column, and we obtain  a Young diagram of a partition, then we denote this partition by $\lambda-h$. 
Otherwise we say that $\lambda-h$ does not exist. Similarly we define $\mu-h$. 
Now 
\begin{equation}\label{eq:modification}
\{ \mu;\lambda\}=\begin{cases}0\in R_{\U(m)}& \text{ if any of }\lambda-h,\mu-h \text{ does not exist};\\
(-1)^{x+y-1}\{\mu-h;\lambda-h\} &\text{ if both }\lambda-h,\mu-h \text{ exist}
\end{cases}
\end{equation}  
where the boundary $h$-strip removed from the Young diagram of $\lambda$ ends in the $x^{\mathrm{th}}$ column, and  the boundary $h$-strip removed from the Young diagram of $\mu$ ends in the $y^{\mathrm{th}}$ column. 
Note that in the special case $p+q=m+1$, i.e. when $h=0$, the outcome of the above rule is $\{ \mu;\lambda\}=-\{ \mu;\lambda\}$, which implies $\{ \mu;\lambda\}=0$ 
whenever $p+q=m+1$. 
 
\begin{example} 
We draw the Young diagram of the partitions $(4,1,1,1)$ and $(3,3,2,1)$ with a boundary $3$-strip (formed by the boxes denoted by $\star$), and the Young diagram of these partitions with a boundary $2$-strip:  

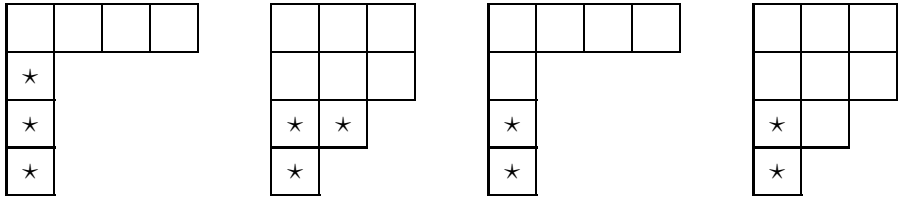
\vspace{0.5cm}
 \begin{ytableau}
 {}&{} &{} &{} \\
\star \\ \star  \\ 
\star
\end{ytableau} \qquad
\begin{ytableau}
{}& {}&{} \\
 {}&{} &{}  \\ 
\star & \star \\
\star 
\end{ytableau} \qquad
 \begin{ytableau}
 {}&{} &{} &{} \\
{}   \\ 
\star \\
\star
\end{ytableau} \qquad
\begin{ytableau}
{}& {}&{} \\
 {}&{} &{}  \\ 
\star & {} \\
\star 
\end{ytableau} 

\medskip
\begin{align*}\text{Thus we have }\{(4,1,1,1);(3,3,2,1)\}=\{(4);(3,3)\}\in R_{U(4)} \\ 
\text{whereas }\{(4,1,1,1);(3,3,2,1)\}=0\in R_{U(5)}. \end{align*}   
\end{example} 

For a triple of partitions $\lambda,\mu,\nu\in \Pi_m$ denote by $c^{\lambda}_{\mu,\nu}$ the associated Littlewood-Richardson coefficient (cf. e.g. \cite[A.8]{fulton-harris}).  It is determined by the equality 
\[\{ 0;\mu\}\cdot \{ 0;\nu\}=\sum_{\lambda\in \Pi_m}c^{\lambda}_{\mu,\nu}\{ 0;\lambda\}\in R_{\U(m)}.\] 
The equality $\{ 0;\mu\}\cdot \{ 0;\nu\}=\{ 0;\nu\}\cdot \{ 0;\mu\}$ implies 
$c^{\lambda}_{\mu,\nu}=c^{\lambda}_{\nu,\mu}$. 

\begin{theorem} [King \cite{king}]\label{thm:king_branching}
For any $[\lambda]\in\tau_{\OO(2m)}$ we have 
\begin{equation}
\label{eq:king-branching}
[\lambda]\downarrow^{\OO(2m)}_{\U(m)}=
\sum_{\nu,\mu,\xi,(2\beta)'\in\Pi_m}
c^{\lambda}_{\mu,\nu}c^{\mu}_{\xi,(2\beta)'}\{ \xi;\nu\} \in R_{\U(m)}
\end{equation} 
where $(2\beta)'$ stands for a partition such that all columns of its Young diagram 
have even length (i.e. the transpose of a partition with even parts).  
\end{theorem} 
The above statement can be found in \cite[p. 440, (4.21)]{king}; the notation $B$ used there is explained at \cite[p. 435-436]{king}.   
For fixed $\xi,\nu,\beta\in \Pi_m$ we have 
\begin{align*}
\sum_{\lambda,\mu\in \Pi_m}c^{\mu}_{\xi,\nu}c^{\lambda}_{\mu,\beta} \{ 0;\lambda\}
&=(\{ 0;\xi\}\cdot \{ 0;\nu\})\cdot \{ 0;\beta\} \\ 
&=(\{ 0;\xi\}\cdot \{ 0;\beta\})\cdot \{ 0;\nu\}
=\sum_{\lambda,\mu\in \Pi_m}c^{\mu}_{\xi,\beta}c^{\lambda}_{\mu,\nu}
 \{ 0;\lambda\}. 
\end{align*} 
It follows that for each $\lambda,\xi,\nu,\beta\in \Pi_m$ we have 
\[\sum_{\mu\in \Pi_m}c^{\mu}_{\xi,\nu}c^{\lambda}_{\mu,\beta}
=\sum_{\mu\in \Pi_m}c^{\mu}_{\xi,\beta}c^{\lambda}_{\mu,\nu}.\] 
So \eqref{eq:king-branching} can be rewritten as 
\begin{equation}\label{eq:2-king-branching}
[\lambda]\downarrow^{\OO(2m)}_{\U(m)}=
\sum_{\nu,\mu,\xi,(2\beta)'\in\Pi_m}
c^{\mu}_{\xi,\nu}c^{\lambda}_{\mu,(2\beta)'}\{ \xi;\nu\} \in R_{\U(m)}. 
\end{equation} 
We note that in the special case when 
$\lambda,\xi,\nu\in \Pi_{\lfloor \frac m2\rfloor}$, 
the same formula for the multiplicity  of $\{\xi;\nu\}$ as a summand in 
$[\lambda]\downarrow^{\OO(2m)}_{\U(m)}$ appears also in \cite[Section 2.3.1]{howe-tan-willenbring}.    

For the complexification of the symmetric tensor power $\Sb^d(\R^{2m})$ of the defining 
$\OO(2m)$-module $\R^{2m}$ we have 
\begin{equation}\label{eq:symmpower-decomp}
\overline{\Sb_{\C}^d(\R^{2m})}=[d]+[d-2]+[d-4]+\cdots,
\end{equation} 
where $d\in\mathbb{N}_0$ is identified with $(d,0,\dots,0)\in \Pi_m$ 
(see e.g. \cite[Section 5.2.3]{goodman-wallach}). 

\begin{lemma}\label{lemma:harmonic-decomp}  
For $m\ge 2$ we have 
\begin{equation*}
[d]\downarrow^{\OO(2m)}_{\U(m)}=\sum_{i+j=d}\{ j;i\}  
\end{equation*} 
where in order to simplify the notation we write 
\[\{j;i\}:=\{(j,0,\dots,0);(i,0,\dots,0)\}.\] 
\end{lemma} 

\begin{proof} 
By the Littlewood-Richardson Rule (see for example 
\cite[p. 498, Theorem]{procesi} or \cite[p.456, A.8]{fulton-harris}) we have that 
if $c^{\lambda}_{\alpha,\beta}\neq 0$ for some $\lambda,\alpha,\beta\in \Pi_m$, 
then $\alpha_1\le\lambda_1,\dots,\alpha_m\le \lambda_m$, and similarly $\beta_1\le\lambda_1,\dots,\beta_m\le\lambda_m$. Moreover, $\sum_{i=1}^m\lambda_i=
\sum_{i=1}^m \alpha_i+\sum_{i=1}^m\beta_i$. 
Apply \eqref{eq:2-king-branching} for the special case $\lambda=(d)$. 
By the above remark, $c^{(d)}_{\mu,(2\beta)'}\neq 0$ holds only if $\beta=(0)$ and $\mu=(d)$, and in this case $c^{(d)}_{\mu,(2\beta)'}=1$. So  \eqref{eq:2-king-branching} 
reduces to 
$[d]\downarrow^{\OO(2m)}_{\U(m)}=\sum_{\xi,\nu\in \Pi_m}c^{(d)}_{\xi,\nu}$. 
Again by the above remark, $c^{(d)}_{\xi,\nu}\neq 0$ holds if and only if 
$\xi=(j)$ and $\nu=(d-j)$ for some $j\in\{0,\dots,d\}$, and $c^{(d)}_{(j);(d-j)}=1$ 
by Pieri's rule (see e.g. \cite[p.455, A.7]{fulton-harris}). 
\end{proof} 

\begin{lemma}\label{lemma:hookremoval} 
Fix non-negative integers $i,j$, and assume that $m\ge 2$. 
\begin{itemize}
\item[(i)] Suppose 
$\xi -h=(j,0^{m-1})$ and $\nu-h=(i,0^{m-1})$, 
where $\xi\in \Pi_p^+$, $\nu\in \Pi_q^+$, $p,q\le m$, and $h=p+q-m-1$. 
Then $j>0$, $i>0$, 
\[\xi=(j,1^{m-1}), \quad \nu=(i,1^{m-1}), \]and in this case $\{\xi-h;\nu-h\}=-\{ j;i\}.$
\item[(ii)] There is no $\xi\in \Pi_m$ and positive integer $h$ such that 
$\xi-h=(j,1^{m-1})$. 
\end{itemize}
\end{lemma} 

\begin{proof} (i) If $\xi-h=(j,0^{m-1})$, then all boxes of the Young diagram of $\xi$ below the first row must belong to the boundary $h$-strip that we remove to get the Young diagram of 
$(j,0^{m-1})$. The inequality 
\[p-1\le h=p+q-m-1\] 
follows, implying that $q\ge m$. Thus we have $q=m$ and $p-1=h$. 
Similarly, $p=m$ and $h=m-1$, so the boundary $h$-strip removed from the Young diagram of $\xi$ (respectively $\nu$) consists of the boxes in the first column beginning from the second row. Moreover, if $\xi=(j,1^{m-1})$ and $\nu=(i,1^{m-1})$, then applying  
\eqref{eq:modification}, we have $x=y=1$ in the formula, hence the modification rule 
\eqref{eq:modification} says $\{\xi-h;\nu-h\}=-\{ j; i\}$.

(ii) Suppose $\xi-h=(j,1^{m-1})$ for some $\xi\in \Pi_m$. 
Unless $h=0$, to get the Young diagram of $\xi-h$ we must remove the bottom box in the first column of the Young diagram of $\xi$. Thus $\xi-h$ has at most $m-1$ non-zero parts, 
whereas $(j,1^{m-1})$ has $m$ non-zero parts. 
\end{proof} 


\section{The space of valuations as a $\U(m)$-module} 

Throughout this section we assume $m\ge 2$. 
Introduce the notation 
\[\langle \lambda\rangle:= [\lambda]\downarrow^{\OO(2m)}_{\SO(2m)}.\]

For $\lambda\in \Pi_m$ with $\lambda_m=0$ the restriction 
$[\lambda]\downarrow^{\OO(2m)}_{\SO(2m)}$ is 
an irreducible $\SO(2m)$-module  with highest weight $\lambda$, 
whereas  for $\mu\in\Pi_m$ with ${\mu_m>0}$ the restriction 
$[\mu]\downarrow^{\OO(2m)}_{\SO(2m)}$ is the direct sum of two non-isomorphic irreducible $\SO(2m)$-modules  with highest weight $\mu$ and $(\mu_1,\dots,\mu_{m-1},-\mu_m)$ 
(cf. \cite[p. 422, Theorem]{procesi}). 
Furthermore, for $k\in \{0,1,\dots,m\}$ set 
\[\ell(k):=\min\{k,2m-k\}.\]

\begin{theorem}\label{thm:alesker-bernig-schuster} 
(Alesker, Bernig, Schuster \cite[Theorem 1]{alesker-bernig-schuster}) 
The $\SO(2m)$-module $\Val_{k,\C}$ admits the  decomposition 
\[\overline{\mathrm{Val}_{k,\C}}=\sum_{(g,2^h)\in \Pi_{\ell(k)}, \ g\neq 1}
\langle (g,2^h)\rangle\]
where $(g,2^h)$ stands for the partition $(g,2,\stackrel{h)}{\ldots},2)$, and where  
the notation introduced in Section~\ref{sec:modification} is extended in the obvious way to locally finite dimensional $\SO(2m)$-representations in which each irreducible $\SO(2m)$-module has finite multiplicity. 
\end{theorem}  

Note that for any $\lambda\in\Pi_m$ we have 
\[\langle\lambda\rangle \downarrow^{\SO(2m)}_{\U(m)}=[\lambda] \downarrow^{\OO(2m)}_{\U(m)},\] 
whence 
\[\overline{\mathrm{Val}_{k,\C}}\downarrow^{\OO(2m)}_{\U(m)}=
\sum_{(g,2^h)\in \Pi_{\ell(k)},\ g\neq 1}
[(g,2^h)]\downarrow^{\OO(2m)}_{\U(m)} \ \in R_{\U(m)}.\] 

\begin{proposition}\label{prop:table_3} 
Table~1 below gives all pairs $([\lambda],\{j;i\})$ where 
$\lambda\in \Pi_m$ is of the form $\lambda=(g,2^h)$, $g\neq 1$, $j\le i$, such that 
$\{ j;i\}$  has 
non-zero coefficient (given in the table) in the expansion \eqref{eq:2-king-branching} of 
$[\lambda]\downarrow^{\OO(2m)}_{\U(m)}$: 
\begin{equation*}\label{eq:table_3} \scriptsize
\begin{array}{cc||c|c|c|c|c|c}
 &  &\{ 0;e\}& \{ 0; e\} &\{ 0;e\} & \{ 1;e-1\}  & \{ 1;e-1\} & \{ j;e-j\}  
  \\ 
 &\hspace{-2cm}\mbox{\em \sc Table 1} &e=0 & e= 1 &e\ge 2 & e=2 & e\ge 3 &  e-j\ge  j\ge 2 
 \\ \hline  \hline
[e] & e\ne 1  &1 &-  & 1 & 1 & 1 & 1   \\      
\lbrack (e,2^h) \rbrack & 2\mid h>0
& - & - &   1 & 1 & 2 & 3  \\
\lbrack (e,2^h) \rbrack & 2\nmid h
& - & - & 0 & 1 & 1 & 1  \\
\lbrack (e+2,2^h) \rbrack & 2\nmid h  & 1 &1 & 1  & 1 & 1 & 1   \\
\lbrack (e-2,2^h) \rbrack & 2\nmid h
& - &- &  0  & - & 0 & 1  
\end{array}
\end{equation*} 
\end{proposition}

\begin{proof} 
Recall that by the  Littlewood-Richardson Rule (see for example \cite[p. 498, Theorem]{procesi} or \cite[p.456, A.8]{fulton-harris}),  
$c^{\eta}_{\alpha,\delta}\neq 0$ for $\alpha\in \Pi_p$, $\delta\in \Pi_q$ implies that $\alpha_s\le\eta_s$ and $\delta_s\le\eta_s$ for each $s=1,\dots,m$, 
$\sum_{s=1}^m\eta_s=\sum_{s=1}^m(\alpha_s+\delta_s)$, 
and $\eta\in\Pi_{p+q}$. Therefore if 
$c_{(i,0^{m-1}),(j,0^{m-1})}^{\mu}c^{(g,2^h)}_{\mu,(2\beta)'}$ is non-zero, then 
$\mu\in \{(i+j),(i+j-1,1),(i+j-2,2)\}$, $(2\beta)'=(2^t)$ with $t\in\{h+1,h,h-1\}$ or 
$(2\beta)'=(2^{h-1},1,1)$, and $i+j\in\{g-2,g,g+2\}$. 

Next we indicate the calculation of the value 
$\dim_{\C}(\mathrm{Hom}_{\U(m)}(\{ j, e-j\},[(e,2^h)]))=3$ for $e-j\ge j\ge 2$, $2\mid h>0$. 
By Pieri's rule  (see e.g. \cite[p.455, A.7]{fulton-harris}) we have 
$c^{\mu}_{(e-j),(j)}=1$  for $\mu\in\{(e),\ (e-1,1),\ (e-2,2)\}$. 
It follows  easily from the Littlewood-Richardson rule  that for 
$\mu\in\{(e),\ (e-1,1),\ (e-2,2)\}$, we have 
$c^{(e,2^h)}_{\mu,(2\delta)'}\neq 0$ holds only if $(2\delta)'=(2^h)$, and in this case   $c^{(e,2^h)}_{\mu,(2\delta)'}=1$. We draw the corresponding Littlewood-Richardson tableau: 

\medskip
\begin{ytableau}
\star&\star&$1$&$1$&$1$ \\
\star &\star  \\ 
$1$ & $1$
\end{ytableau}
$\qquad$
\begin{ytableau}
\star&\star&$1$&$1$&$1$\\
\star &\star  \\ 
$1$ & $2$
\end{ytableau}
$\qquad$ 
\begin{ytableau}
\star&\star&$1$&$1$&$1$ \\
\star &\star  \\ 
$2$ & $2$
\end{ytableau}

\medskip

Or by similar considerations, the values 
$\dim_{\C}(\mathrm{Hom}_{\U(m)}(\{ 1; e-1\},[([e+2,2^h)]))=1$ and 
$\dim_{\C}(\mathrm{Hom}_{\U(m)}(\{ 1;e-1\},[(e,2^h)]))=1$ for $2\nmid h$ correspond to the tableaux

\medskip 

\begin{ytableau} 
\star &\star& $1$ &$1$ &$1$  \\
\star &\star \\
\star & \star \\
\star & \star 
\end{ytableau} 
$\qquad$ 
\begin{ytableau}\star &\star& $1$ \\
\star &\star \\
\star & $1$ \\
\star & $2$ 
\end{ytableau} 

\smallskip The remaining entries in Table~1 are calculated similarly.  
\end{proof}

\begin{remark} 
The number in Table~1 in the row labeled by $[\lambda]$ and the column labeled by $\{j;i\}$  
is the coefficient of $\{j;i\}$ in the expansion \eqref{eq:2-king-branching} of $[\lambda]\downarrow^{\OO(2m)}_{\U(m)}$. Note that for certain  $[\lambda]$ and $\{j;i\}$ this number is not equal to the \emph{multiplicity} of the irreducible 
$\U(m)$-module $\{j;i\}$ as a summand in $[\lambda]\downarrow^{\OO(2m)}_{\U(m)}$. 
Indeed, the expansion \eqref{eq:2-king-branching} of $[\lambda]\downarrow^{\OO(2m)}_{\U(m)}$ may contain terms $\{\xi;\nu\}$ where $\{\xi;\nu\}$ does not belong to 
$\tau_{\U(m)}$ (namely when $\xi\in \Pi_p^+$, $\nu\in \Pi_q^+$ with $p+q>m$). 
When such an element $\{\xi;\nu\}$ is expanded in terms of the basis $\tau_{\U(m)}$ of  $R_{U(m)}$ using King's modification rules \eqref{eq:modification}, basis elements of the form $\{j;i\}$ may appear with non-zero integer coefficient.  
The necessary modifications in Table~1 in order to get the \emph{multiplicity} of 
$\{j;i\}$ as a summand in $[\lambda]\downarrow^{\OO(2m)}_{\U(m)}$ are taken into account in Proposition~\ref{prop:table_4} below. 
\end{remark} 

\goodbreak
\begin{proposition}\label{prop:table_4} \mbox{}
\begin{itemize}\item[(i)] 
For $h<m-1$, Table~1 gives the  
non-zero   multiplicities of summands of the form $\{j;i\}$  with $j\le i$ in the  
$\U(m)$-modules  
$[(g,2^h)]\downarrow^{\OO(2m)}_{\U(m)}$.   
\item[(ii)] 
The non-zero multiplicities of the summands of the form $\{j;i\}$  with $j\le i$ in the  
$\U(m)$-modules  $[(g,2^{m-1})]\downarrow^{\OO(2m)}_{\U(m)}$    
are given in Table~2 below: 
\begin{equation*}\label{eq:table_4} \scriptsize
\begin{array}{cc||c|c|c|c|c|c}
 &  &\{ 0;e\}& \{ 0; e\} &\{ 0;e\} & \{ 1;e-1\}  & \{ 1;e-1\} & \{ j;e-j\}  
  \\  
 & \hspace{-2.5cm}\mbox{\em \sc Table 2} &e=0 & e= 1 &e\ge 2 & e=2 & e\ge 3 &  e-j\ge  j\ge 2 
 \\ \hline  \hline

\lbrack (e,2^{m-1}) \rbrack & 2\nmid m
& - & - &   1 & 0 & 1 & 2  \\
\lbrack (e,2^{m-1}) \rbrack & 2\mid m
& - & - & 0 & 0 & 0 & 0  \\
\lbrack (e+2,2^{m-1}) \rbrack & 2\mid m  & 1 &1 & 1  & 1 & 1 & 1   \\
\lbrack (e-2,2^{m-1}) \rbrack & 2\mid m
& - &- &  0  & - & 0 & 1  
\end{array}
\end{equation*} 
\end{itemize}
\end{proposition}

\begin{proof} 
 To get the multiplicity of $\{ j;i\}$ in $[(g,2^h)]\downarrow^{\OO(2m)}_{\U(m)}$ 
 we may start with the expansion \eqref{eq:2-king-branching}, which expresses 
$[(g,2^h)]\downarrow^{\OO(2m)}_{\U(m)}$ as a non-negative integer linear combination 
of symbols $\{\xi;\nu\}$, where $\xi,\nu\in \Pi_m$. The terms $\{\xi;\nu\}$ not belonging to $\tau_{\U(m)}$ need to be rewritten in terms of the basis $\tau_{U(m)}$ by a possibly iterated use of King's modification rules \eqref{eq:modification}. Then we can collect the coefficient of 
$\{j;i\}$ in $[(g,2^h)]\downarrow^{\OO(2m)}_{\U(m)}$ with respect to the basis $\tau_{\U(m)}$.  
By \eqref{eq:modification} and Lemma~\ref{lemma:hookremoval} (i), (ii), apart from $\{j;i\}$ the only $\{\xi;\nu\}$ whose expansion with 
respect to the basis $\tau_{\U(m)}$ involves a basis element of the form $\{j;i\}$ is 
$\{(j,1^{m-1});(i,1^{m-1})\}$, and the coefficient of $\{j;i\}$ in $\{(j,1^{m-1});(i,1^{m-1})\}$ with respect to the basis $\tau_{\U(m)}$ is $-1$. 
 Now $\{(j,1^{m-1});(e-j,1^{m-1})\}$ occurs with   coefficient $1$ in the expansion \eqref{eq:2-king-branching} of $[(e,2^{m-1})]\downarrow^{\OO(2m)}_{\U(m)}$, and does not occur in $[(g,2^h)]\downarrow^{\OO(2m)}_{\U(m)}$ for $h<m-1$ or if  $g\neq e$. 
Consequently, by Proposition~\ref{prop:table_3},  the non-zero   multiplicities of $\{ j;e-j\}$  (with $j\le e-j$) in the  
$\U(m)$-modules  
$[(g,2^h)]\downarrow^{\OO(2m)}_{\U(m)}$  
are given by Table~1 for $h<m-1$ or $g\neq e$, whereas for $h=m-1$,  
 $g=e$ and $j\ge 1$ we need to subtract $1$ from the corresponding entry of Table~1, and that is how we obtained   
Table~2.  
\end{proof}

Now we are in position to give the multiplicities of the irreducible $\U(m)$-modules $\{j;i\}$ as a summand of $\Val_{k,\C}$. 

\begin{theorem}\label{prop:table_5} 
The table  below gives the non-zero multiplicities of summands of the form  
$\{ j;i\}$ with 
$i\ge j$ in $\mathrm{Val}_{k,\C}\downarrow^{\OO(2m)}_{\U(m)}$ for $0\le k\le 2m$: 
\begin{equation*} \scriptsize
\begin{array}{c||c|c|c|c|c|c}
   \mathrm{Val}_{k,\C} & \{ 0,0\} & \{ 0;1\} & \{ 0; e\}  & \{ 1;1\} & \{ 1; e-1\} & \{ j; e-j\} \\ 
 & & & e\ge 2 &  & e\ge 3 &  e-j\ge  j\ge 2 
 \\ 
 \hline \hline 
 k\in\{0,2m\} & 1 & 0 & 0& 0 & 0 & 0 \\
  1\le \ell(k)<m &1+ \lfloor \frac{\ell(k)}{2} \rfloor & \lfloor \frac{\ell(k)}{2} \rfloor & \ell(k) & \ell(k)+\lfloor\frac{\ell(k)}2 \rfloor & 2\ell(k)-1 & 3\ell(k)-2
\\ 
k=m  &  1+\lfloor \frac m2 \rfloor &  \lfloor \frac m2 \rfloor & m-1 
& m+\lfloor\frac m2 \rfloor -1& 2m-2 & 3m-3
\end{array}
\end{equation*}
\end{theorem}

\begin{proof} 
By Theorem~\ref{thm:alesker-bernig-schuster} and Proposition~\ref{prop:table_4}  we have 
that for $e\ge 3$, and $0<\ell(k)<m$,  the multiplicity of $\{ 1;e-1\}$ in $\mathrm{Val}_{k,\C}\downarrow^{\OO(2m)}_{\U(m)}$ is 
\[1+2\cdot \#\{h\colon 0< h\le \ell(k)-1,\ 2\mid h\}+
2\cdot \#\{h\colon  h\le \ell(k)-1,\ 2\nmid h\}=2\ell(k)-1, \] 
whereas the multiplicity of  $\{ 1;1\}$ in $\mathrm{Val}_{k,\C}\downarrow^{\OO(2m)}_{\U(m)} 
$ is 
\[1+\#\{h\colon 0<h\le \ell(k)-1,\ 2\mid h\}
+2\cdot \#\{h\colon  h\le \ell(k)-1,\ 2\nmid h\}=\ell(k)+\left\lfloor \frac{\ell(k)}2\right\rfloor.\] 
Similar considerations yield the rest of the table. 
\end{proof} 

\begin{proposition}
\label{prop:table_6}
The dimensions of the spaces of $\U(m)$-module homomorphisms 
from $[e]\downarrow^{\OO(2m)}_{\U(m)}$ to $\mathrm{Val}_{k,\C}$  
($e\geq 0$; $0\le k\le 2m$)  are the following (for greater legibility, the notation $\ell:=\ell(k)$ is used in some places): 
\begin{equation*}
\begin{array}{cc||c}
 &  & \dim_{\C}(\mathrm{Hom}_{\U(m)}( [e],\mathrm{Val}_{k,\C})) \\ 
\hline \hline 
 e=0, & \ell(k)>0 & 1+\lfloor \frac{\ell}{2}\rfloor \\
 e=1, & \ell(k)>0 & 2\lfloor \frac{\ell}{2} \rfloor \\
 e=2, & 0<\ell(k)<m & 3\ell+\lfloor \frac{\ell}{2} \rfloor \\
 e\ge 3, & 0<\ell(k)<m &  3\ell e-3\ell-2e+4  \\  
e=2, & k=m  & 3m+\lfloor\frac{m}{2}\rfloor -1 \\ 
e\ge 3, & k=m & 3me-3m-3e+5 \\ 
e=0 & \ell(k)=0 & 1 \\
e\neq 0 & \ell(k)=0 & 0
\end{array}
\end{equation*} 
\end{proposition} 

\begin{proof} 
By Lemma~\ref{lemma:harmonic-decomp}, the $U(m)$-module $[e]\downarrow^{\OO(2m)}_{\U(m)}$ 
is multiplicity free, and it is the direct sum of the modules $\{j;e-j\}$. 
Therefore by Schur's Lemma,  
$\dim_{\C}(\mathrm{Hom}_{\U(m)}([e],\mathrm{Val}_{k,\C}))$  
equals the sum of the multiplicities of the irreducible $\U(m)$-module direct summands 
$\{ j;e-j\}$ of 
$\mathrm{Val}_{k,\C}\downarrow^{\OO(2m)}_{\U(m)}$. 
These sums of multiplicities can be easily determined using 
Theorem~\ref{prop:table_5}, and taking into account that the multiplicity of $\{ j;i\}$ in $\mathrm{Val}_k\downarrow^{\OO(2m)}_{\U(m)}$ 
equals the multiplicity of $\{ i; j\}$  by \eqref{eq:2-king-branching} and by $c_{\mu,\nu}^\lambda=c_{\nu,\mu}^\lambda$. 
\end{proof} 

\begin{proposition}
\label{prop:table_7}
The dimensions of the spaces 
$\mathrm{Hom}_{\U(m)}( \Sb_{\C}^d(\R^{2m}),\mathrm{Val}_{k,\C})$ 
($d\geq 0$; $0\le k\le 2m$)  
are the following (we shall use the notation 
$f:=\lfloor \frac d2 \rfloor$ and $\ell=\ell(k)$): 

 \begin{equation*}
\begin{array}{cc||c}
 & & \dim_{\C}(\mathrm{Hom}_{\U(m)}( \Sb_{\C}^d(\R^{2m}),\mathrm{Val}_{k,\C}))\\ \hline \hline 
d=0 &  & 1+\lfloor \frac{\ell}2\rfloor \\
d=2f>0, & \ell=0 & 1 \\
d=2f>0, & 1\le \ell<m & 3\ell f^2+2\lfloor \frac{\ell}{2} \rfloor-2f^2+2f+1 \\
d=2f>0, & \ell=m & 3mf^2+2\lfloor \frac m2 \rfloor -3f^2+2f+1 \\
d=2f+1, & \ell=0 & 0 \\
d=2f+1, & 1\le \ell<m & 3\ell f^2+3\ell f+2\lfloor \frac{\ell}{2}\rfloor-2f^2 \\
d=2f+1, & \ell=m &  3mf^2+3mf + 2\lfloor \frac m2\rfloor-3f^2-f
\end{array} 
\end{equation*} 
\end{proposition}  

\begin{proof} 
The $\U(m)$-module $\Sb_{\C}^d(\R^{2m})$ is multiplicity free by 
\eqref{eq:symmpower-decomp} and  by  Lemma~\ref{lemma:harmonic-decomp}. 
Therefore by Schur's Lemma and by \eqref{eq:symmpower-decomp} we have 
\begin{equation*}
\dim_{\C}(\mathrm{Hom}_{\U(m)}( \Sb_{\C}^d(\R^{2m}),\mathrm{Val}_{k,\C})
=\sum_{p=0}^{\lfloor \frac d2\rfloor}\dim_{\C}(\mathrm{Hom}_{\U(m)}([d-2p],\mathrm{Val}_{k,\C})). 
\end{equation*} 
Thus the statement  follows easily from Proposition~\ref{prop:table_6}. 
\end{proof} 

\begin{proofof}{Theorem~\ref{thm:dimensions}} The result is an immediate consequence 
of Proposition~\ref{prop:table_7} by \eqref{eq:real_complex} and
\eqref{eq:tensor_hom}. 
\end{proofof} 

\begin{remark}\label{remark:odd-even}  \mbox{}
\begin{itemize} 
\item[(i)] The case $d=0$ in Proposition~\ref{prop:table_7}, i.e. the  equality 
\[\dim_{\C}((\mathrm{Val}_{k,\C})^{\U(m)})=1+\left\lfloor \frac{\ell(k)}2 \right\rfloor\]  
is due to Alesker \cite[Theorem 6.1]{Ale01}. 
\item[(ii)] The special case $d=1$ in Proposition~\ref{prop:table_7}, namely that  
\[\dim_{\C}((\mathrm{Val}_{k,\C}\otimes \Sb^1_\C(\R^{2m})^{\U(m)}))=2\left\lfloor \frac{\ell(k)}2\right\rfloor,\]  
is due to Wannerer \cite{wannerer}, see also  
\cite[Theorem  6.14]{schuster}.  
\end{itemize}  
\end{remark}

\section{A basis for vector valued valuations in terms of area measures}
\label{secVecbasis}
In this section we construct a $\C$-vector space basis of 
$(\Val\otimes_\R \C^m)^{\U(m)}$, the complex vector space of $\U(m)$-equivariant translation invariant vector valued valuations on $\C^m$ for $m\ge 2$.

In the scalar valued case, Bernig and Fu \cite{Ber-Fu} constructed a basis of $\Val^{\U(m)}$ consisting of the so-called hermitian intrinsic volumes $\mu_{k,q}$, defined for $0\leq k\leq 2m$ and $0,k-m\leq q\leq \frac{k}2$. 
These valuations are even and hence characterized by their Klain function \cite{Klain2000}. The Klain function of an even valuation $\varphi\in \Val_k(\V)$ is a function $\mathrm{Kl}_\varphi$ on the $k$-Grassmannian $\mathrm{Gr}_k(\V)$ given by $\varphi(A)=\mathrm{Kl}_\varphi(E)\mathrm{vol}_k(A)$ for $A\subset E\in \mathrm{Gr}_k(\V)$.

For $k\leq m$, the  Klain function of $\mu_{k,q}$ is
\begin{equation}\label{klain_mu}
 \mathrm{Kl}_{\mu_{k,q}}(E) =\sum_{i=q}^{\lfloor\frac{k}{2}\rfloor} (-1)^{i+q}{i \choose q}\sigma_i(\cos^2\theta_1,\ldots, \cos^2\theta_{\lfloor k/2\rfloor})
\end{equation}
where $\theta_i$ is the $i^\mathrm{th}$ elementary symmetric function, and $\theta_1,\ldots,\theta_{\lfloor k/2\rfloor}$ are the K\"ahler angles of the $k$-dimensional linear subspace $E$. These angles are characterized as follows. Let $\psi_E$ be the endomorphism of $E$ that maps $u\in E$ to the orthogonal projection of $\sqrt{-1}u$ to $E$. Then $\psi_E$  has eigenvalues $$\pm\sqrt{-1}\cos\theta_1,\ldots, \pm\sqrt{-1}\cos\theta_{\lfloor k/2\rfloor},$$ plus a zero eigenvalue if $k$ is odd. For $k>m$, 
\begin{equation}\label{klain_mu_2}
 \mathrm{Kl}_{\mu_{k,q}}(E)=\mathrm{Kl}_{\mu_{2m-k,m-k+q}}(E^\bot).
\end{equation}

On the other hand, Wannerer \cite{wannerer14} introduced the space $\Area(\V)$ of smooth area measures on a  euclidean vector space $\V$. These are certain translation invariant valuations on $\V$, taking values on the space of signed measures of the unit sphere $S(\V)$. Thus, if $\Phi\in \Area(\V)$ and $A\subset \V$ is a convex body, then $\Phi(A,\cdot)$ is a signed measure on $S(\V)$.  The globalization map $\mathrm{glob}\colon \mathrm{Area}(\V)\to \Val(\V)$  and the centroid map $C\colon \mathrm{Area}(\V)\to \Val(\V)\otimes \V$ are then defined by
\[
 \mathrm{glob}({\Phi})(A)=\Phi(A,S(\V)),\qquad C(\Phi)(A)=\int_{S(\V)} u\ d\Phi(A,u).
\] 

Given a linear subspace $E\subset \V$ there exists a restriction map $r$ from $\mathrm{Area}(\V)$ to $\mathrm{Area}(E)$ characterized as follows. Given a Borel set $U\subset S(V)$, let $\overline U=(U+E^\bot)\cap S(\V)$. The restriction of $\Phi\in\mathrm{Area}(\V)$ is given by
\begin{equation}\label{restriction}
r(\Phi)(A,U)=\Phi(A,\overline U),\qquad A\in\K(E), U\subset S(E). 
\end{equation}

For $\V=\C^m$,  the space $\mathrm{Area}_k^{\U(m)}$ of $k$-homogeneous $\U(m)$-invariant smooth area measures was described in \cite{wannerer14}. We will need the following.
\begin{proposition}[\cite{wannerer14}]\label{deltas}
 Given $0\leq k < 2m$, there exists a family $\Delta_{k,q}\in\Area_k^{\U(m)}$ with $0,k-m\leq q\leq \frac{k}2$ such that \begin{enumerate}
 \item[(i)] $\mathrm{glob}(\Delta_{k,q})=\mu_{k,q}$
 \item[(ii)] for every polytope $P$, and every Borel set $U\subset S(\C^m)$
  \begin{equation}\label{angular}
   \Delta_{k,q}(P, U)=\sum_{F\in\mathcal F_k}\mathrm{Kl}_{\mu_{k,q}}(\vec F) \frac{\mathrm{vol}_{2m-k-1}(N(P,F)\cap U)}{\mathrm{vol}_{2m-k-1}(S^{2m-k-1})} \mathrm{vol}_k(F)
  \end{equation}
where $\mathcal F_k$ is the set of $k$-dimensional faces, $N(P,F)$ is the set of outer unit normal vectors to $P$ at points of $F$, and $\vec F$ is the $k$-dimensional linear subspace parallel to $F$.
 \item[(iii)] \label{it:res_deltas}The restriction $r\colon \mathrm{Area}(\C^{m+l})\to \mathrm{Area}(\C^m)$ corresponding to the inclusion $\C^m\to\C^{m+l}$ fulfills $r(\Delta_{k,q})=\Delta_{k,q}$ if $q\geq k-m$.
\end{enumerate}
\end{proposition}
Given a $p$-dimensional real subspace  $E\subset \C^m$ and  the corresponding restriction map $r$, it follows  from \eqref{angular} that 
\begin{equation}\label{eq:centroid_restriction}
C(r(\Delta_{k,q}))(A)=c_{m,p,k}C(\Delta_{k,q})(A),\qquad A\in\K(E),
\end{equation}
for $c_{m,p,k} \neq 0$ depending only on $m,p,k$.

It  was shown in \cite{wannerer14} that the family $C(\Delta_{k,q})$  with $0,k-m< q\leq \frac{k}2$ is $\R$-linearly independent. Since we already know that 
$\dim_\C((\mathrm{Val}_k\otimes_{\R}\C^m)^{\U(m)})=\lfloor \frac{\ell(k)}{2}\rfloor$ 
(see Remark~\ref{remark:odd-even} (ii)), 
to prove Theorem~\ref{thm:Vecbasis} we only need to show that the above family is in fact $\C$-linearly independent. 

\smallskip
\begin{proofof}{Theorem~\ref{thm:Vecbasis}}
Let us first assume $k\geq m$. For $1\leq r\leq \lfloor\frac{2m-k}{2}\rfloor$ consider the following element of $\mathrm{Area}_k^{\U(m)}$: 
\[
\Psi_{k,r}=\sum_{i=r}^{\lfloor\frac{2m-k}2\rfloor} {i\choose r} \Delta_{k,k-m+i}.
\]
By \eqref{klain_mu}  and \eqref{klain_mu_2} we have 
\[
\mathrm{Kl}_{\mathrm{glob} (\Psi_{k,r})}(F)= \sigma_r(\cos^2\theta_1,\ldots, \cos^2\theta_{ \lfloor {2m-k\over 2}\rfloor}).
\]  where the $\theta_i$ refer to the K\"ahler angles of $F^\bot$.

We show next that the vector valuations $C(\Psi_{k,r})$, and hence the $C(\Delta_{k,q})$ with $q>k-m$, are linearly independent elements of the complex vector space 
$(\Val_k
\otimes_\R \C^m)^{\U(m)}$. Since this space has dimension  $\lfloor\frac{2m-k}2\rfloor$, the statement will follow.

Given $k+1-m\leq q\leq \frac{k+1}2$, we may consider 
\[
E=\C^q\oplus\R^{k-2q+1}
\]
which is a subspace of $\C^m$ with $\dim_\R E=k+1$. Let $T_q$ be the $(k+1)$-dimensional simplex in $E$ with vertices 
\[
0,e_1,  \sqrt{-1}e_1,\ldots, e_q,  \sqrt{-1}e_q,e_{q+1},\ldots, e_{k-q+1}
\]
We proceed to compute $C(\Psi_{k,r})(T_q)$. 

The linear spaces parallel to the $k$-faces $F$ of $T_q$ with normal vectors $u\in\{e_1, \sqrt{-1}e_1,\ldots, e_q, \sqrt{-1}e_q\}$ belong to the $\U(m)$-orbit of $\C^{q-1}\oplus\R^{k-2q+2}=(\C^{m-k+q-1}\oplus\R^{k-2q+2})^\bot$. Hence the K\"ahler angles $\theta_i$ of $F^\bot$ are given by
\[
\cos\theta_1=\cdots=\cos\theta_{m-k+q-1}=1,\quad \cos\theta_{m-k+q}=\cdots=\cos\theta_{\lfloor \frac{2m-k}2\rfloor}=0
\]
which yields
\[
\mathrm{Kl}_{\mathrm{glob}(\Psi_{k,r})}(F)={m-k+q-1 \choose r}.
\]
Similarly, the linear spaces parallel to the $k$-faces $F$ of $T_q$ with normal vectors $u\in\{e_{q+1},\ldots, e_{k-q+1}\}$ belong to the $\U(m)$-orbit of $\C^{q}\oplus\R^{k-2q}=(\C^{m-k+q}\oplus\R^{k-2q})^\bot$. Hence the K\"ahler angles $\theta_i$ of  $F^\bot$ are given by
\[
\cos\theta_1=\cdots=\cos\theta_{m-k+q}=1,\quad \cos\theta_{m-k+q+1}=\ldots=\cos\theta_{\lfloor \frac{2m-k}2\rfloor}=0
\]
which yields
\[
\mathrm{Kl}_{\mathrm{glob}(\Psi_{k,r})}(F)={m-k+q \choose r}.
\]
It remains to consider the $k$-face $F$ of $T_q$ opposite to $0$. Its outer unit  normal vector in $E$ is 
\[v=\frac{1}{\sqrt{k+1}}(e_1+\sqrt{-1}e_1+\cdots +e_q+\sqrt{-1}e_q+e_{q+1}+\cdots +e_{k-q+1})\] 
and its normal space $F^\bot$ in $\C^m$ is spanned by 
\[v, \sqrt{-1}e_{q+1},\ldots,  \sqrt{-1}e_{k-q+1}, e_{k-q+2}, \sqrt{-1}e_{k-q+2},\ldots, e_m, \sqrt{-1}e_m.\] 
The K\"ahler angles $\theta_i$ of $F^\bot$ can be obtained from the eigenvalues of $\psi_{F^\bot}$. The outcome of this computation is
\[
\cos\theta_1=\cdots=\cos\theta_{m-k+q-1}=1,\quad\cos\theta_{m-k+q}=\sqrt\frac{k-2q+1}{k+1} 
\]
and $\cos\theta_i=0$ for $i>m-k+q$, which yields
\[
\mathrm{Kl}_{\mathrm{glob}(\Psi_{k,r})}( F)={m-k+q-1\choose r}+{m-k+q-1 \choose r-1}\frac{k-2q+1}{k+1}.
\]
Finally, by \eqref{angular}   and \eqref{eq:centroid_restriction}
\begin{align*}
c_{m,k+1,k}&C(\Psi_{k,r})(T_q)=\frac{1}{2\cdot k!}{m-k+q-1\choose r}(-e_1-\sqrt{-1}e_1-\ldots-e_q-\sqrt{-1}e_q)\\
&+\frac{1}{2\cdot k!}{m-k+q \choose r}(-e_{q+1}-\ldots-e_{k-q+1})\\
&+\frac{\sqrt{k+1}}{2\cdot k!}\left[{m-k+q-1\choose r}+{m-k+q-1 \choose r-1}\frac{k-2q+1}{k+1}\right]v\\
&={m-k+q-1\choose r-1}\frac{1}{2(k+1)!}\Big[(k-2q+1)(e_1+\sqrt{-1}e_1+\ldots +\sqrt{-1}e_q)\\&-2q(e_{q+1}+\ldots+e_{k-q+1})\Big]
\end{align*}

Suppose now that  $\sum_i a_i C(\Psi_{k,i})=0$ for $a_1,\ldots, a_{\lfloor \frac{2m-k}2\rfloor}\in\C$. For each $i=1,\ldots,\lfloor\frac{2m-k}2\rfloor$, pick $q=k-m+i$. By the above expression we get \[
C(\Psi_{k,r})(T_q)=0 \Leftrightarrow r> i.\] By induction this shows $a_1=\cdots=a_{\lfloor\frac{2m-k}2\rfloor}=0$. 

The case $k<m$ can be deduced from the previous one as follows. Suppose
\begin{equation}\label{eq:case2}\sum_{i=1}^{\lfloor k/2\rfloor} a_i C(\Delta_{k,i})=0\end{equation} with $a_i\in\C$. Consider the inclusion $\C^k\subset \C^m$. By item (iii) in Proposition \ref{deltas} and \eqref{eq:centroid_restriction}, the relation \eqref{eq:case2} holds in $\C^k$. Hence, we can apply the previous case to conclude $a_i=0$ for all $i$.
\end{proofof}

\bigskip
\noindent {\bf Acknowledgement: } We would like to heartily thank Semyon Alesker, Andreas Bernig, Daniel Hug, Endre Szab\'o, Thomas Wannerer for enlightening discussions.


\end{document}